\documentclass[12pt]{amsart}

\usepackage{amsthm,amsfonts,amsmath,amssymb,amscd,latexsym,epsfig,mathrsfs,yfonts,marvosym}
\usepackage[usenames]{color}
\usepackage{graphicx}
\usepackage[all]{xy}
\usepackage{epsfig}
\usepackage{epic}
\usepackage{eepic}
\usepackage{dsfont}\let\mathbb\mathds
\usepackage[english]{babel}

\DeclareMathAlphabet\oldmathcal{OMS}        {cmsy}{b}{n}
\SetMathAlphabet    \oldmathcal{normal}{OMS}{cmsy}{m}{n}
\DeclareMathAlphabet\oldmathbcal{OMS}       {cmsy}{b}{n}
 
\usepackage{eucal}

\usepackage[active]{srcltx}

\newtheorem{theorem}{Theorem}[section]
\newtheorem{lemma}[theorem]{Lemma}

\newtheorem{corollary}[theorem]{Corollary}
\newtheorem{definition}[theorem]{Definition}

\newtheorem{def/prop}[theorem]{Definition/Proposition}

\newenvironment{example}{\medskip \refstepcounter{theorem}
\noindent  {\bf Example \thetheorem}.\rm}{\,}
\newenvironment{remark}{\medskip \refstepcounter{theorem}
\newcommand     {\comment}[1]   {}
\newcommand{\mute}[2] {}
\newcommand     {\printname}[1] {}

\noindent  {\bf Remark \thetheorem}.\rm}{\,}
\newtheorem*{ack}{Acknowledgements}

\def\<{\langle}

\def\>{\rangle}

\def\BOne{{\mathchoice {\rm 1\mskip-4mu l} {\rm 1\mskip-4mu l}
                          {\rm 1\mskip-4.5mu l} {\rm 1\mskip-5mu l}}}

\def\fract#1#2{\raise4pt\hbox{$ #1 \atop #2 $}}
\def\decdnar#1{\phantom{\hbox{$\scriptstyle{#1}$}}
\left\downarrow\vbox{\vskip15pt\hbox{$\scriptstyle{#1}$}}\right.}

\def\bbc{{\mathbb C}}

\def\bbn{{\mathbb N}}

\def\bbp{{\mathbb P}}
\def\bbq{{\mathbb Q}}
\def\bbr{{\mathbb R}}

\def\bbt{{\mathbb T}}

\def\bbz{{\mathbb Z}}

\def\gra{\alpha}
\def\grb{\beta}

\def\grg{\gamma}

\def\grl{\lambda}

\def\gro{\omega}

\def\grs{\sigma}

\def\grD{\Delta}

\def\bfl{{\bf l}}
\def\bfm{{\bf m}}
\def\bfn{{\bf n}}

\def\bfv{{\bf v}}
\def\bfw{{\bf w}}

\def\calb{{\mathcal B}}

\def\calo{{\mathcal O}}

\def\calf{{\mathcal F}}

\def\calh{{\mathcal H}}
\def\cali{{\mathcal I}}

\def\calk{{\mathcal K}}
\def\call{{\mathcal L}}

\def\calo{{\mathcal O}}

\def\cals{{\oldmathcal S}}

\def\calt{{\mathcal T}}

\def\la#1{\hbox to #1pc{\leftarrowfill}}
\def\ra#1{\hbox to #1pc{\rightarrowfill}}
\def\calz{{\oldmathcal Z}}

\def\gz{{\mathfrak z}}

\def\gB{{\mathfrak B}}

\def\gF{{\mathfrak F}}

\def\gL{{\mathfrak L}}
\def\gM{{\mathfrak M}}

\def\gO{{\mathfrak O}}

\def\hook{\mathbin{\hbox to 6pt{%
                 \vrule height0.4pt width5pt depth0pt
                 \kern-.4pt
                 \vrule height6pt width0.4pt depth0pt\hss}}}

\def\12{\xi_{k_1,k_2}}
\def\m5{M^5_{k_1,k_2}}

\begin{document}

\title{Sasaki-Einstein Metrics on a class of 7-Manifolds}

\author[Charles Boyer]{Charles P. Boyer}
\address{Charles P. Boyer, Department of Mathematics and Statistics,
University of New Mexico, Albuquerque, NM 87131.}
\email{cboyer@math.unm.edu} 
\author[Christina T{\o}nnesen-Friedman]{Christina W. T{\o}nnesen-Friedman}
\address{Christina W. T{\o}nnesen-Friedman, Department of Mathematics, Union
College, Schenectady, New York 12308, USA }
\email{tonnesec@union.edu}
\thanks{The authors were partially supported by grants from the
Simons Foundation, CPB by (\#519432), and CWT-F by (\#422410)}
\date{\today}

\begin{abstract}
In this note we give an explicit construction of Sasaki-Einstein metrics on a class of simply connected 7-manifolds with the rational cohomology of the 2-fold connected sum of $S^2\times S^5$. The homotopy types are distinguished by torsion in $H^4$.

\end{abstract}

\maketitle
\vspace{-7mm}

\section*{Introduction}\label{intro}
Sasaki-Einstein metrics on 7-manifolds continue to play an important role in M-theory as well as black hole physics \cite{Spa10,GLPS17,MoTo18}. An important reason for this is that Sasaki-Einstein manifolds admit supersymmetry and are used in the AdS-CFT correspondence. For such reasons it seems important to have a large list of explicit examples of Sasaki-Einstein 7-manifolds that can be used as possible models. Since these SE metrics are toric, general existence of such metrics in well known \cite{FOW06,Leg16}. What is new here is an explicit construction of such metrics, their relation with Bott manifolds (orbifolds), and the topological description of the 7-manifolds. We give an explicit construction of toric Sasaki-Einstein (SE) 7-manifolds which can be represented as $S^1$ orbibundles over 2-twist stage 3 Bott orbifolds.
All of these are obtained by adding orbifold structures to certain stage 3 Bott manifolds which were studied in \cite{BoCaTo17}. The 7-manifolds all have the rational cohomology of the 2-fold connected sum $2(S^2\times S^5)$ and are generalizations of the SE 7-manifolds given by Theorem 1.2 of \cite{BoTo14NY}.

\begin{ack}
The authors would like to thank David Calderbank for his comments and his interest in this work. In particular, the second author would like to express her gratitude to David for our many helpful conversations about orbifold metric constructions.
\end{ack} 

\section{Stage Three Bott Towers and Orbifolds}
Following \cite{GrKa94} and \cite{BoCaTo17} we consider {\it Bott towers} which in arbitrary dimension is represented by a lower triangular unipotent matrix $A$ over $\bbz$. Here we deal only with stage 3 Bott towers, so the matrix $A$ in \cite{GrKa94,BoCaTo17} takes the form 
\begin{equation}\label{A3matrix}
A=\begin{pmatrix}
1 & 0 & 0 \\
a & 1 & 0 \\
b & c & 1
\end{pmatrix}
\end{equation}
with $a,b,c\in \bbz$. The Bott manifold can be realized as the quotient of $S^3\times S^3\times S^3$ by the $\bbt^3$ action
\begin{equation}\label{Bottaction}
(z_j^0,z_j^\infty)_{j=1}^3\mapsto (t_jz_j^0,\prod_{i=1}^3t_i^{A^i_j}z_j^\infty).
\end{equation} 
The quotient $M_3$ which is called a {\it Bott manifold} can be represented as a sequence, called a {\it Bott tower}
\begin{equation}\label{Botttower}
M_3 \xrightarrow{\pi_3}M_2
\xrightarrow{\pi_2} M_1\xrightarrow{\pi_1}M_0= \mathit{pt},
\end{equation}
where the jth $S^3$ is written as $|z_j^0|^2 +|z_j^\infty|^2=1$.
Then $M_k$ is the
compact complex manifold arising as the total space of the $\bbc\bbp^1$ bundle
$\pi_k\colon \bbp(\BOne \oplus \call_k) \to M_{k-1}$.
At each stage we have \emph{zero} and \emph{infinity sections} $\grs_k^0\colon
M_{k-1}\to M_k$ and $\grs_k^\infty\colon M_{k-1}\to M_k$ which respectively
identify $M_{k-1}$ with $\bbp(\BOne \oplus 0)$ and $\bbp(0\oplus \call_k)$.
We consider these to be part of the structure of the Bott tower
$(M_k,\pi_k,\grs_k^0,\grs_k^\infty)_{k=1}^n$.  Here in our case $M_3=M_3(a,b,c)$ is a stage 3 Bott manifold, $M_2(a)$ is a Hirzebruch surface $\calh_a = \bbp(\BOne \oplus \calo(a))\ra{1.8} \bbc\bbp^1$, and $M_1=\bbc\bbp^1$. $M_3(a,b,c)$ can be viewed as the total space of $\bbc\bbp^1$ bundle over the Hirzebruch surface
$\calh_a$, and also a bundle of Hirzebruch surfaces over $\bbc\bbp^1$ with
fiber $\calh_c$. Bott towers form the object set $\calb\calt_0$ of a groupoid whose morphisms $\calb\calt_1$ are biholomorphisms \cite{BoCaTo17}, and elements of the quotient space $\calb\calt_0/\calt_1$ are identfied with biholomorphism classes of Bott manifolds. Since Bott manifolds are toric, they are described by a fan, and it follows from the Bott tower description \eqref{Botttower} that  
the fan of the Bott tower $M_3(a,b,c)$ is described by the primitive collections (cf. \cite{CoLiSc11})
\begin{equation}\label{primcol}
\{v_1,u_1\},\quad \{v_2,u_2\},\quad \{v_3,u_3\}
\end{equation}
with normal vectors $u_1=-v_1-av_2-bv_3, u_2=-v_2-cv_3,u_3=-v_3$, and thus has the combinatorial type of a cube.
The symmetry group of a cube is the Coxeter group
$BC_3\cong \mathrm{Sym}_3\ltimes\bbz_2^3$ where $\mathrm{Sym}_3$ is the symmetric group on 3 letters. However, not all elements of $BC_3$ are induced by equivalences. We refer to \cite{BoCaTo17} and references therein for details.

The structure of Bott towers implies the existence of 3 pairs of $\bbt^3$ invariant divisors $\{D_{v_j},D_{u_j}\}$ which are the zero and infinity sections of $M_j\xrightarrow{\pi_j}M_{j-1}$ with $j=1,2,3$. Thus, elements of the subgroup $\bbz_2^3$ are induced by the {\it fiber inversion} maps that interchange the zero and infinity sections, so these equivalences always exist. However, elements of $\mathrm{Sym}_3$ are induced by equivalences only in special cases (see Lemma 1.11 and Example 1.6 of \cite{BoCaTo17} for details).

Recall (Definition 1.6 of \cite{BoCaTo17}) that the {\it holomorphic twist} of an $n$ dimensional Bott tower is the number $t\in\{0,\ldots,n-1\}$ of holomorphically nontrivial $\bbc\bbp^1$ bundles in the tower. So for $n=3$ there are only 3 possibilities $t=0,1,2$. Of course, $t=0$ is the well understood product $\bbc\bbp^1\times \bbc\bbp^1\times \bbc\bbp^1$, whereas $t=1$ leads to the Koiso-Sakane case, so here we restrict our attention to $t=2$. For the $t=2$ case we have a nontrivial holomorphic bundle over a Hirzebruch surface $\calh_a$ with $a\neq 0$ whose fiber is $\bbc\bbp^1$. The case of interest to us can be obtained as an $S^3_\bfw$-join with the $Y^{p,q}$ structures of \cite{GMSW04a} on $S^2\times S^3$, that is, $(S^2\times S^3)\star_{l_1,l_2}S^3_\bfw$. So the stage 3 Bott manifold has an additional orbifold structure which we now describe.

\subsection{Invariant Divisors}\label{divsec}
Here we consider the $\bbt^3$ invariant divisors $D_{v_i},D_{u_i}$ defined by the normal vectors $v_i,u_i$ with equivalences
$$D_{v_1}\sim D_{u_1},\qquad D_{v_2}\sim aD_{u_1}+D_{u_2},\qquad D_{v_3}\sim bD_{u_1}+cD_{u_2}+D_{u_3}.$$
We have 4 sets of distinguished invariant bases of the Chow group $A_2(M_3)$ of invariant divisor classes
\begin{gather}\
\label{bases1} \{[D_{u_1}],[D_{u_2}],[D_{u_3}]\}, \\
\{[D_{u_1}],[D_{u_2}],[D_{v_3}]\}, \\
\{[D_{u_1}],[D_{v_2}],[D_{u_3}]\}, \\
\label{bases4} \{[D_{u_1}],[D_{v_2}],[D_{v_3}]\}. 
\end{gather}
This gives rise to the 4 sets of dual bases of cohomology classes in $H^2(M_3,\bbz)$, viz
\begin{gather}\
\label{dualbases1} \{x_1,x_2,x_3\}, \\
\{x_1,x_2,y_3\}, \\
\{x_1,y_2,x_3\}, \\
\label{dualbases4} \{x_1,y_2,y_3\}, 
\end{gather}
where $x_i$ is dual to $[D_{u_i}]$ and $y_i$ is dual to $[D_{v_i}]$. Both the ample and K\"ahler cones can be easily worked out, see Example 3.3 of \cite{BoCaTo17}.

\subsection{Bott Orbifolds and log pairs}
We are interested in these Bott manifolds, but with an additional special orbifold structure along the $\bbt^3$-invariant divisors $D_{v_i},D_{u_i}$. 
The orbifold structure on $M_3(a,b,c)$ that we are interested in is given by the log pair $(M_3(a,b,c),\grD_\bfm)$ where $\grD_\bfm$ is the branch divisor
\begin{equation}\label{bradiv}
\Bigl(1-\frac{1}{m_1^0}\Bigr)D_{v_1}+\Bigl(1-\frac{1}{m_1^\infty}\Bigr)D_{u_1} +\Bigl(1-\frac{1}{m_2^0}\Bigr)D_{v_2}+\Bigl(1-\frac{1}{m_2^\infty}\Bigr)D_{u_2} +\Bigl(1-\frac{1}{m_3^0}\Bigr)D_{v_3}+\Bigl(1-\frac{1}{m_3^\infty}\Bigr)D_{u_3}
\end{equation}
where $m_j^0,m_j^\infty\in\bbz^+$ are the ramification indices. We define $\grD_{\bfm_3}=\grD_\bfm$ and 
\begin{eqnarray*}
\grD_{\bfm_2}&=&\Bigl(1-\frac{1}{m_1^0}\Bigr)D_{v_1}+\Bigl(1-\frac{1}{m_1^\infty}\Bigr)D_{u_1} +\Bigl(1-\frac{1}{m_2^0}\Bigr)D_{v_2}+\Bigl(1-\frac{1}{m_2^\infty}\Bigr)D_{u_2} \\ 
\grD_{\bfm_1}  &=&\Bigl(1-\frac{1}{m_1^0}\Bigr)D_{v_1}+\Bigl(1-\frac{1}{m_1^\infty}\Bigr)D_{u_1}.
\end{eqnarray*}
Clearly we have 
$$\grD_{\bfm_1}\subset \grD_{\bfm_2} \subset \grD_{\bfm_3}.$$
From \cite{GrKa94} or \cite{BoCaTo17} one easily sees that
\begin{lemma}\label{orbquot} 
We have the sequence of Bott towers of log pairs
\begin{equation}\label{Bottorbitower}
\bigl(M_3(a,b,c),\grD_{\bfm_3}\bigr)\fract{\pi_3}{\ra{2.2}}\bigl(M_2(a),\grD_{\bfm_2}\bigr)\fract{\pi_2}{\ra{2.2}} \bigl(M_1,\grD_{\bfm_1})\fract{\pi_1}{\ra{2.2}}(\{pt\},\emptyset),
\end{equation}
where $M_2(a)=\calh_a$ is a Hirzebruch surface, $M_1=\bbc\bbp^1$, and $\pi_i$ is the natural projection. 
\end{lemma}

The invariant branch divisors are related to the section maps by
$$D_{v_3}=\grs_2^0(\calh_a),~D_{u_3}=\grs_2^\infty(\calh_a),~D_{v_2}=\pi_3^{-1}(\grs_1^0(\bbc\bbp^1)),~D_{u_2}=\pi_3^{-1}(\grs_1^\infty(\bbc\bbp^1))$$ 
and 
$$D_{v_1}=(\pi_2\circ\pi_3)^{-1}(\grs_1^0(\{pt\})),~ D_{u_1}=(\pi_2\circ\pi_3)^{-1}(\grs_1^\infty(\{pt\})).$$
We denote by $\gB\gO_A$ the set of all such Bott orbifolds. 

\subsection{The Orbifold First Chern Class}
We can now compute the orbifold canonical divisor $K^{orb}$ and dually the orbifold first Chern class.
We compute the orbifold first Chern class $c_1^{orb}$ in the $\{x_j\}$ basis for any n-dimensional Bott manifold $M_n(A)$;
\begin{eqnarray}\label{c1orb}
c_1^{orb}(M_\bfn(A),\grD_\bfm) &=&c_1(M_\bfn(A)) -\sum_{i=1}^n\Bigl(\bigl(1-\frac{1}{m_i^0}\bigr)y_i +\bigl(1-\frac{1}{m_i^\infty}\bigr)x_i\Bigr) \notag \\
                                             &=& \sum_{i=1}^n(x_i+y_i)-\sum_{i=1}^n\Bigl(\bigl(1-\frac{1}{m_i^0}\bigr)y_i +\bigl(1-\frac{1}{m_i^\infty}\bigr)x_i\Bigr) \notag \\
                              &=&\sum_{j=1}^n\bigl(\frac{1}{m_j^0}y_j+\frac{1}{m_j^\infty}x_j\bigr)  =\sum_{i=1}^n\Bigl((\frac{1}{m_i^0}+\frac{1}{m_i^\infty})x_i+\sum_{j=1}^{i-1}\frac{A^j_i}{m_i^0}x_j\Bigr) \notag \\
                              &=& \sum_{j=1}^{n-1}\Bigl(\frac{1}{m_j^0}+\frac{1}{m_j^\infty}+\sum_{i=j+1}^n\frac{A^j_i}{m_i^0}\Bigr)x_j +\Bigl(\frac{1}{m_n^0}+\frac{1}{m_n^\infty}\Bigr)x_n.
\end{eqnarray}
In dimension n there are $2^{n-1}$ invariant bases in which to compute $c_1^{orb}$. This becomes much more manageable for $n=3$. We have $c_1^{orb}(M_3(a,b,c),\grD_\bfm)$ in the four bases \eqref{dualbases1}-\eqref{dualbases4}, respectively
\begin{eqnarray}\label{c1orb1}
\bigl(\frac{1}{m_1^0}+\frac{1}{m_1^\infty}+\frac{a}{m_2^0}+\frac{b}{m_3^0}\bigr)x_1+  \bigl(\frac{1}{m_2^0}+\frac{1}{m_2^\infty}+\frac{c}{m_3^0}\bigr)x_2 +\bigl(\frac{1}{m_3^0}+\frac{1}{m_3^\infty}\bigr)x_3, \\
\bigl(\frac{1}{m_1^0}+\frac{1}{m_1^\infty}+\frac{a}{m_2^0}-\frac{b}{m_3^\infty}\bigr)x_1+  \bigl(\frac{1}{m_2^0}+\frac{1}{m_2^\infty}-\frac{c}{m_3^\infty}\bigr)x_2 +\bigl(\frac{1}{m_3^0}+\frac{1}{m_3^\infty}\bigr)y_3, \\
\bigl(\frac{1}{m_1^0}+\frac{1}{m_1^\infty}-\frac{a}{m_2^\infty}+\frac{b-ac}{m_3^0}\bigr)x_1+  \bigl(\frac{1}{m_2^0}+\frac{1}{m_2^\infty}+\frac{c}{m_3^0}\bigr)y_2 +\bigl(\frac{1}{m_3^0}+\frac{1}{m_3^\infty}\bigr)x_3, \\
\bigl(\frac{1}{m_1^0}+\frac{1}{m_1^\infty}-\frac{a}{m_2^\infty}-\frac{b-ac}{m_3^\infty}\bigr)x_1+  \bigl(\frac{1}{m_2^0}+\frac{1}{m_2^\infty}-\frac{c}{m_3^\infty}\bigr)y_2 +\bigl(\frac{1}{m_3^0}+\frac{1}{m_3^\infty}\bigr)y_3. \label{c1orb4}                                                 
\end{eqnarray}
Note also that as a cohomology class $c_1^{orb}(M_3(a,b,c),\grD_\bfm)$ as an element of $H^{1,1}(M_n(a,b,c),\bbr)$ makes perfect sense for all $m_j^0,m_j^\infty\in\bbr^+$.
We shall make use of this fact shortly. 

Equations \eqref{c1orb1}-\eqref{c1orb4} implies
\begin{lemma}\label{orbFano}
Let $(M_3(a,b,c),\grD_\bfm)$ be a Bott orbifold. Then the following are equivalent:
\begin{enumerate}
\item $(M_3(a,b,c),\grD_\bfm)$ is log Fano, 
\item $c_1^{orb}(M_3(a,b,c),\grD_\bfm)$ lies in the K\"ahler cone $\calk(M_3(a,b,c))$,
\item the inequalities
\begin{gather}\ \notag
\frac{1}{m_1^0}+\frac{1}{m_1^\infty}+\frac{a}{m_2^0}+\frac{b}{m_3^0}>0,\quad  \frac{1}{m_2^0}+\frac{1}{m_2^\infty}+\frac{c}{m_3^0}>0 \\
\frac{1}{m_1^0}+\frac{1}{m_1^\infty}+\frac{a}{m_2^0}-\frac{b}{m_3^\infty}>0,\quad  \frac{1}{m_2^0}+\frac{1}{m_2^\infty}-\frac{c}{m_3^\infty}>0 \notag \\
\frac{1}{m_1^0}+\frac{1}{m_1^\infty}-\frac{a}{m_2^\infty}+\frac{(b-ac)}{m_3^0}>0,\quad  \frac{1}{m_2^0}+\frac{1}{m_2^\infty}+\frac{c}{m_3^0}>0 \notag \\
\frac{1}{m_1^0}+\frac{1}{m_1^\infty}-\frac{a}{m_2^\infty}-\frac{(b-ac)}{m_3^\infty}>0,\quad  \frac{1}{m_2^0}+\frac{1}{m_2^\infty}-\frac{c}{m_3^\infty}>0 \notag
\end{gather}
hold.
\end{enumerate}
\end{lemma}

\section{$M^7$ as the join $Y^{p,q}\star_{l_1,l_2}S^3_\bfw$}
Not every $S^1$ orbibundle over a Bott orbifold can be realized as a join; however, the 2 twist stage 3 Bott orbifolds that we study here can be realized as K\"ahler quotients of the join $Y^{p,q}\star_{l_1,l_2}S^3_\bfw$ where $Y^{p,q}$ are the well known SE structures on $S^2\times S^3$ discovered by the physicists \cite{GMSW04a}. In fact, since $Y^{p,q}$ is itself a join of two $S^3$'s, it is an iterated join of three $S^3$'s which in the terminology of \cite{BHLT16} is completely cone decomposable. Now from Example 6.8 of \cite{BoTo14a} we have
\begin{equation}\label{Ypqjoin}
Y^{p,q}=S^3\star_{l,p}S^3_{\frac{p+q}{l},\frac{p-q}{l}}
\end{equation}
where $l=\gcd(p+q,p-q)$ which equals $2$ if $p,q$ are both odd, and equals $1$ if $p,q$ have opposite parity. Note that we choose the standard SE structure on the lefthand $S^3$ factor, whereas it is the weighted Sasakian structure with weights $(\frac{p+q}{l},\frac{p-q}{l})$ on the righthand $S^3$. 

If we assume that $p>q\geq 1$ are such that $\sqrt{4p^2-3q^2 }\in\bbn$, then the Sasaki-Einstein structure on $Y^{p.q}$ is quasi-regular \cite{GMSW04a}. Indeed, the $\eta$-Einstein structure corresponds to a ray in the so-called $\bfw$ subcone is determined by
co-prime solutions $(v^0_2,v^\infty_2)$ of
\begin{equation}\label{YpqisEinstein}
 \int_{-1}^1 ((v^0_2-v^\infty_2)-(v^0_2+v^\infty_2)\gz)(((p+q)v_2^\infty + (p-q)v^0_2)+((p+q)v_2^\infty - (p-q)v^0_2)\gz)\,d\gz=0
\end{equation}
(from e.g. (68) in \cite{BoTo14a}).  Following Theorem 3.8 in \cite{BoTo14a}, for any choice of quasi-regular ray determined by the
co-prime pair $(v^0_2,v^\infty_2)$, the quotient Hirzebruch orbifold is $(\calh_a,\grD_{\bfm_2})$ where
$\grD_{\bfm_2} = (1-\frac{1}{m_2^0})D_{v_2} +(1-\frac{1}{m_2^\infty})D_{u_2}$ is the branch divisor with 
\begin{equation}\label{step1data}
\begin{array}{ccl}
{\bfm}_2 & = & (m^0_2,m^\infty_2)=m_2(v^0_2,v^\infty_2),\\
\\
m_2& = & \frac{p}{\gcd(p, |\frac{p+q}{l} v^\infty_2-\frac{p-q}{l} v^0_2|)},\\
\\
a & = &\frac{(p+q)m^\infty_2-(p-q)m^0_2}{p}.
\end{array}
\end{equation}

The join $M_{l_1,l_2,\bfw}=Y^{p,q}\star_{l_1,l_2}S^3_\bfw$, where $Y^{p,q}$ has a quasi-regular Sasaki structure as above, can then be obtained as the quotient of the following $\bbt^2$ action on $S^3\times S^3\times S^3$:
\begin{equation}\label{T2act}
(x,u;u_1,u_2;z_1,z_2)\mapsto (x,e^{ipl_2\theta}u;e^{i(l_2m_2v^0_2\phi-(p+q)\theta)}u_1, e^{i(l_2m_2v^\infty_2\phi-(p-q)\theta)}u_2;e^{-il_1w_1\phi}z_1,e^{-il_1w_2\phi}z_2).
\end{equation}
First we notice that without loss of generality we can assume that $\gcd(l_1,m_2)=1$, for otherwise we can redefine $\phi$. So our gcd conditions are $\gcd(l_1,l_2m_2)=1$, $\gcd(w_1,w_2)=1=\gcd(v^0_2,v^\infty_2)$, and $\gcd(p,q)=1$. Note that $p+q$ and $p-q$ can have a common factor if and only if both $p$ and $q$ are odd, in which case the common factor is $2$. However, $M_{l_1,l_2,\bfw}$ may not be a smooth manifold. we have

\begin{lemma}\label{smoothquot}
The join $M_{l_1,l_2,\bfw}=Y^{p,q}\star_{l_1,l_2}S^3_\bfw$, with Sasakian structure on $Y^{p,q}$ given by the Reeb field $\xi_{{\bfv}_2}$ with ${\bfv}_2=(v^0_2,v^\infty_2)$, is a smooth manifold if and only if $\gcd(l_2m_2v^i_2,l_1w_j)=1$ for $i=0,\infty$ and $j=1,2$, where
$m_2= \frac{p}{\gcd(p, |\frac{p+q}{l} v^\infty_2-\frac{p-q}{l} v^0_2|)}$.
\end{lemma}

\begin{proof}
From Proposition 7.6.7 of \cite{BG05} $M_{l_1,l_2,\bfw}$ is smooth if and only if $\gcd(l_1\Upsilon_2,l_2\Upsilon_1)=1$ where $\Upsilon_1$ is the order of $Y^{p,q}$ and $\Upsilon_2$ is the order of $S^3_\bfw$. The latter is $\Upsilon_2=w_1w_2$ with $w_1,w_2$ coprime. The order $\Upsilon_1$ with quasi-regular Reeb field $\xi_{\bfv_2}$
is $\Upsilon_1=m_2v^0_2v^\infty_2$.
\end{proof}

The analysis in Section 3 of \cite{BoTo14a} holds equally well when the manifold $M$ in the join $M\star_{l_1,l_2}S^3_\bfw$ has any quasi-regular Sasakian structure. The major difference is having more complicated computations. For example we need the Fano index of of $Y^{p,q}$. As described above, the quotient of any quasi-regular Sasakian structure in the $\bfw$ subcone of $Y^{p,q}$ is a Hirzebruch orbifold of the form $(\calh_a,\grD_{{\bfm}_2})$. We have
\begin{lemma}\label{Ypqind}
Let $\xi_{{\bfv}_2}$ be a quasiregular Reeb vector field with quotient Hirzebruch orbifold $(\calh_a,\grD_{{\bfm}_2})$ with $a>0$. Then its Fano index $\cali_{\bfv_2}$ is given by 
$$\cali_{\bfv_2}=\gcd(2m_2v^0_2+a)v^\infty_2,v^0_2+v^\infty_2)$$
where $\bfm_2=m_2(v^0_2,v^\infty_2)$ and $v^0_2,v^\infty_2$ are coprime.
\end{lemma}

\begin{proof}
Recall (Definition 4.4.24 of \cite{BG05}) that the Fano index $\cali$ of an orbifold $\calz$ is the largest positive integer $k$ such that $p^*c_1^{orb}/k$ is an element of $H^2_{orb}(\calz,\bbz)=H^2(B\calz,\bbz)$. Now the classifying map $p:B\calz\ra{1.8} \calz$ is an $m_2v^0_2v^\infty_2$-fold cover, and $c_1^{orb}$ is 
\begin{equation}\label{c1orbeqn}
c_1^{orb}=(2+\frac{a}{m_2v^0_2}) x_1+\frac{v^0_2+v^\infty_2}{m_2v^0_2v^\infty_2}x_2=\frac{1}{m_2v^0_2v^\infty_2}\bigl((2m_2v^0_2+a)v^\infty_2 x_1+(v^0_2+v^\infty_2)x_2\Bigr).
\end{equation}
So $p^*c_1^{orb}=(2m_2v^0_2+a)v^\infty_2 p^*x_1+(v^0_2+v^\infty_2)p^*x_2$ from which the result follows.
\end{proof}

From now on we assume that $p>q\geq 1$ are such that $\sqrt{4p^2-3q^2 }\in\bbn$ and $Y^{p,q}$ has the quasi-regular Sasaki-Einstein structure. Thus co-prime $(v^0_2,v^\infty_2)$ are chosen such that  \eqref{YpqisEinstein} is satisfied. We then present a
version of Theorem 3.8 of \cite{BoTo14a} that allows the join $M_{l_1,l_2,\bfw}=Y^{p,q}\star_{l_1,l_2}S^3_\bfw$

\begin{theorem}\label{38thm}
Consider the join $M_{l_1,l_2,\bfw}=Y^{p,q}\star_{l_1,l_2}S^3_\bfw$ where $Y^{p,q}$ has a quasi-regular Sasaki-Einstein structure
determined by the co-prime pair $(v^0_2,v^\infty_2)$ satisfying Equation \eqref{YpqisEinstein}.
Let $\cals_{\bfv_3}=(\xi_{\bfv_3},\eta_{\bfv_3},\Phi_{\bfv_3},g_{\bfv_3})$ be a quasi-regular Sasakian structure that lies in the $\bfw$ subcone of the Sasaki cone with ${\bfv_3}=(v^0_3,v^\infty_3)$ where $v^0_3,v^\infty_3$ are coprime. Then the quotient of $M_{l_1,l_2,\bfw}$ by the $S^1$ action generated by $\xi_{\bfv_3}$ is the Bott orbifold given by the log pair $(M_3(a,b,c),\grD_\bfm)$, where $a$ is determined by \eqref{step1data}, 
$$
\begin{array}{ccccl}
b& = & n\hat{b} & = & n\frac{(2m_2v^0_2+a)v^\infty_2}{\cali_{\bfv_2}}\\
\\
c& = & n\hat{c} &= & n\frac{(v^0_2+v^\infty_2)}{\cali_{\bfv_2}},
\end{array}
$$
$n  =  l_1\frac{w_1 v^\infty_3-w_2 v^0_3}{\gcd(|w_1 v^\infty_3-w_2 v^0_3|,l_2)}$, 
$\cali_{\bfv_2}$ is given by Lemma \ref{Ypqind},
and $\grD_\bfm$ is the branch divisor
$$\Bigl(1-\frac{1}{m_1^0}\Bigr)D_{v_1}+\Bigl(1-\frac{1}{m_1^\infty}\Bigr)D_{u_1} +\Bigl(1-\frac{1}{m_2^0}\Bigr)D_{v_2}+\Bigl(1-\frac{1}{m_2^\infty}\Bigr)D_{u_2} +\Bigl(1-\frac{1}{m_3^0}\Bigr)D_{v_3}+\Bigl(1-\frac{1}{m_3^\infty}\Bigr)D_{u_3}$$
with $(m^0_1,m^\infty_1)=(1,1)$, $(m^0_2,m^\infty_2)$ given by \eqref{step1data}, and
$$(m^0_3,m^\infty_3)=m_3(v^0_3,v^\infty_3) =\frac{l_2}{\gcd(|w_1 v^\infty_3-w_2 v^0_3|,l_2)}(v^0_3,v^\infty_3).$$
\end{theorem}

\begin{proof}
We can follow the proof of Theorem 3.8 of \cite{BoTo14a} with the caveat that $N$ is a Hirzebruch orbifold $(\calh_a,\grD_{\bfm_2})$. As in Equation (3) of \cite{BoTo14a} we have the commutative diagram
\begin{equation}\label{s2comdia}
\begin{CD}  S^3\times S^3\times S^3 &&& \\
                           \decdnar{} &&& \\
                           Y^{p,q}\times S^3_\bfw &&& \\
                          &\searrow\pi_L &&& \\
                          \decdnar{\pi_{2}} && Y^{p,q}\star_{l_1,l_2,\bfw}S^3_\bfw &\\
                          &\swarrow\pi_1 && \\
                      (\calh_a,\grD_{\bfm_2})\times\bbc\bbp^1[\bfw] &&& 
\end{CD}
\end{equation}
where the $\pi$s are the obvious projections, and the orbifold $(\calh_a,\grD_{\bfm_2})$ is the quotient by the locally free $S^1$ action on $Y^{p,q}$ generated by the quasi-regular Reeb vector field $\xi_{\bfm_2}$ where $\bfm_2=m_2(v_2^0,v_2^\infty)$ and $a$ is given in Equations \eqref{step1data}. The holomorphic line bundle $L_n=L^n$ now becomes the holomorphic line orbibundle with $L$ determined by the ``primitive'' K\"ahler class in $H^2((\calh_a,\grD_\bfm),\bbq)$, namely
\begin{equation}\label{primKah}
\frac{c_1^{orb}(\calh_a,\grD_{\bfm_2})}{\cali_{\bfv_2}}=\frac{1}{m_2v^0_2v^\infty_2}\frac{\bigl((2m_2v^0_2+a)v^\infty_2 x_1+(v^0_2+v^\infty_2)x_2\Bigr)}{\cali_{\bfv_2}}.
\end{equation}
We make note of the fact that $Y^{p,q}\ra{1.8} (\calh_a,\grD_{\bfm_2})$ is an $m_2v_2^0v_2^\infty$-fold covering map.

Now as in Section 3.5 of \cite{BoTo14a} we want to describe the base orbifold $B_{\bfl,\bfv,\bfw}$ of the $S^1$ orbibundle generated by the locally free action of the quasi-regular Reeb vector field $\xi_{\bfm_3}$ of the SE structure on the join $Y^{p,q}\star_{\bf \ell}S^3_\bfw$. We see that the analysis of Section 3.5 of \cite{BoTo14a} goes through verbatim through Remark 3.7 with $M=Y^{p,q}$ and $N=(\calh_a,\grD_{\bfm_2})$. In particular, from Lemma 3.6 of \cite{BoTo14a} we obtain the base orbifold       $B_{{\bf \ell},\bfv,\bfw}\approx (B_{{\bf \ell},1,\bfw'},\grD)$ with 
$$\grD=\Bigl(1-\frac{1}{m_3^0}\Bigr)D_{v_3}+\Bigl(1-\frac{1}{m_3^\infty}\Bigr)D_{u_3},~~\bfw'=(v_3^\infty w_1, v_3^0 w_2),$$
and from Theorem 3.8 of \cite{BoTo14a} we have
\begin{equation}\label{m3eqns}
\bfm_3=m_3(v_3^0,v_3^\infty),~s=\gcd(|w_1v_3^\infty-w_2v_3^0|,l_2),~l_2=sm_3,~n=\frac{l_1}{s}(w_1v_3^\infty-w_2v_3^0).
\end{equation}
Then from the proof of Theorem 3.8 we see that the quotient is the total space of the projective orbibundle $\bbp(\BOne\oplus L^n)$ over $(\calh_a,\grD_{\bfm_2})$ whose invariant divisors are generally branch divisors of an orbifold. But then using Lemma \ref{orbquot} this is precisely a stage 3 Bott orbifold $(M_3(a,b,c),\grD_\bfm)$ for some $b,c$ and where $a$ is given in Equations \eqref{step1data}. The fact that $Y^{p,q}$ as a join has the form of Equation \eqref{Ypqjoin} with the standard regular Sasakian structure on the first $S^3$ implies that the ramification indices $\bfm_1=(1,1)$. 

It remains to check the equations for $b$ and $c$. For this we make use of an orbifold version of Proposition 1.5 in \cite{BoCaTo17}. Explicitly, we have
\begin{lemma}\label{orbilinebundle}
The pullback of $c_1(L^n)$ of the orbifold line bundle $L^n$ to $M_3(a,b,c)$ is $bx_1+cx_2$ where $n$ is given in Equations \eqref{m3eqns}.
\end{lemma}

We know that $L^n$ is the nth tensor product of the line orbibundle $L$ which is determined by the K\"ahler class 
\begin{equation}\label{KahB}
\frac{c_1^{orb}(\calh_a,\grD_{\bfm_2})}{\cali_{\bfv_2}}=\frac{\bigl((2m_2v^0_2+a)v^\infty_2 x_1+(v^0_2+v^\infty_2)x_2\Bigr)}{m_2v^0_2v^\infty_2\cali_{\bfv_2}}.
\end{equation}
Now the projection $p:(M_3(a,b,c),\grD)\ra{1.8} (\calh_a,\grD_{\bfm_2})$ is a $m_2v^0_2v^\infty_2$-fold covering map. So pulling back we have $c_1(L^n)=nc_1(L)=np^*\left(\frac{c_1^{orb}(\calh_a,\grD_{\bfm_2})}{\cali_{\bfv_2}}\right)$. The equations for $b$ and $c$ then follow by equating coefficients in this and in Lemma \ref{orbilinebundle}.
\end{proof}

\begin{remark}\label{qdiv}
The Poincar\'e dual to $c_1^{orb}(M_3(a,b,c),\grD_\bfm)$ is a $\bbq$-divisor on $M_3(a,b,c)$ which is ample when $c_1^{orb}(M_3(a,b,c),\grD_\bfm)$ is positive. Such a class gives a polarization to the orbifold $(M_3(a,b,c),\grD_\bfm)$, and a $\bbt^3$ invariant orbifold 2-form representing $c_1^{orb}(M_3(a,b,c),\grD_\bfm)$ gives an orbifold K\"ahler metric $g_{a,b,c,\bfm}$ on $M_3(a,b,c)$.
\end{remark}

\begin{remark}
Note that the real cohomology class $c_1^{orb}(M_3(a,b,c),\grD_\bfm)$ makes perfect sense for $m_j^0,m_j^\infty\in\bbr^+$, and we denote the set of all such classes by $\gB\gL_A$. In this case $(M_3(a,b,c),\grD_\bfm)$ can be understood as having K\"ahler metrics with conical singularities along the corresponding $\bbr$-divisors $D_{v_i}$ and $D_{u_i}$ with cone angle $\frac{2\pi}{m_i^0}$ and $\frac{2\pi}{m_i^\infty}$, respectively \cite{Don12,ChDoSu15I}. By this we mean that there is a K\"ahler metric $\gro$ which is smooth on $M_n(a,b,c)\setminus \grD_\bfm$ which extends to $M_n(a,b,c)$ as a closed positive $(1,1)$ current satisfying certain uniformity requirements. See Definition 1.3 of \cite{ChDoSu15I} for the precise statement. For arbitrary $\bfm$ we say that $c_1^{orb}(M_3(a,b,c),\grD_\bfm)$ represents a cone singularity along the divisor $\grD_\bfm$. The rational entries in the interval $(0,1)$ are related to the so-called `ramifolds' \cite{RoTh11}. As in this reference we shall also assume hereafter that $m_j^0,m_j^\infty\in (0,\infty)$. We denote by $\gL\gF_A$ the subset of $\gB\gL_A$ whose K\"ahler metrics are log Fano and have cone singularities along the divisor $\grD_\bfm$. Then the natural map from $\gL\calf_A$ to the K\"ahler cone $\calk(M_3(a,b,c))$ is surjective, and the conclusion of Lemma \ref{orbFano} holds for all $(M_3(a,b,c),\grD_\bfm)\in \gL\gF_A$.
\end{remark}

\begin{remark}
We consider the action of the affine monoid $\gM(\bbr)$ on $(\bbr^+)^6$ defined by the affine linear map 
\begin{equation}\label{affmap}
(m_j^0,m_j^\infty)\mapsto (\grl_j^0m_j^0+a_j^0,\grl_j^\infty m_j^\infty+a_j^\infty)=(\tilde{m}_j^0,\tilde{m}_j^\infty)
\end{equation} 
with $1\leq \grl_j<\infty$ and $0\leq a_j<\infty$. Restricting to the positive integers gives an action of the submonoid $\gM(\bbz)$ on $(\bbz^+)^6$. 
One easily checks that this induces an action of $\gM(\bbr)$ on $\gB\gL$ that leaves the subset $\gL\gF_A$ invariant for all $\grl_j^0,\grl_j^\infty\in [1,\infty)$ and $a_i^0,a_j^\infty\in [0,\infty)$ sending $c_1^{orb}(M_3(a,b,c),\grD_\bfm)$ to $c_1^{orb}(M_3(a,b,c),\grD_{\tilde{\bfm}})$.  
\end{remark}

\section{The Topology of $M^7=Y^{p,q}\star_{l_1,l_2}S^3_\bfw$}
It is important to remember that generally the topology of a join depends on the choice of Sasakian structure (through its Reeb vector field) of each factor.
We assume that $(p,q)$ are relatively prime with $1\leq q<p$ and that $l_1,l_2,w_1,w_2$ are chosen such that $M^7$ is smooth. We show first that our Sasaki 7-manifolds $M^7$ have the rational cohomology of the 2-fold connected sum $(S^2\times S^5)\# (S^2\times S^5)$. The integer cohomology groups are only distinguished by torsion in $H^4$. Moreover, the torsion depends on the choice of quasi-regular Sasakian structures on $Y^{p,q}$ and $S^3$. For the most generality we choose arbitrary quasi-regular Sasakian structure in the so-called $\bfw$ subcone of the Sasaki cones for both $Y^{p,q}$ and $S^3$ (of course, the $\bfw$ cone of $S^3$ is its entire Sasaki cone).

First we note that any quotient of a quasi-regular Reeb vector field $\xi_\bfm$ in the $\bfw$ cone of $Y^{p,q}$ has the form of a Hirzebruch orbifold $(\calh_a,\grD_\bfm)$.
Moreover, $Y^{p,q}$ is itself the join $S^3\star_{l_1,p}S^3_\bfw$ with $\bfw=(\frac{p+q}{l_1},\frac{p-q}{l_1})$. Here $l_1=2$ if $p,q$ are both odd, and $l_1=1$ is $p,q$ have opposite parities. In any case the relation with the ramification indices is $\bfm=m(v^0,v^\infty)$ with $v^0,v^\infty$ coprime and $m=p$.

The purpose of this section is to prove
\begin{theorem}\label{maintopthm}
Let $Y^{p,q}$ have a quasi-regular Sasakian structure with Reeb vector field $\xi_\bfm$. Then the 7-manifolds $M^7=Y^{p,q}\star_{l_1,l_2}S^3_\bfw$ have the rational cohomology of the connected sum $(S^2\times S^5)\# (S^2\times S^5)$. Furthermore, the only torsion that occurs is $H^4(M^7,\bbz)\approx \bbz_{v^0v^\infty m^2l_2^2}\oplus \bbz_{w_1w_2l_1^2}$ .
\end{theorem}

We begin with some lemmas.
\begin{lemma}\label{exhomseq}
The 7-manifolds $M^7=Y^{p,q}\star_{l_1,l_2}S^3_\bfw$ satisfy the following conditions: 
\begin{enumerate}
\item $H_1(M^7,\bbz)=\pi_1(M^7)=0$,
\item $\pi_2(M^7)=\bbz^2$, 
\item $H^2(M^7,\bbz)=H_2(M^7,\bbz)=\bbz^2$,
\item $H^3(M^7,\bbz)$ is torsion free.
\item $b_3(M^7)=b_4(M^7)$ is even.
\end{enumerate}
\end{lemma}

\begin{proof}
From the long exact homotopy sequence for the fibration 
\begin{equation}\label{t2fibration}
\bbt^2\ra{1.8} S^3\times S^3\times S^3\ra{1.8} M^7
\end{equation}
we conclude that $M^7$ is simply connected and that $\pi_2(M^7)=\bbz^2$. Thus, by Hurewicz $H_2(M^7,\bbz)=\bbz^2$ which implies (4) by universal coefficients, and then by item (1) $H^2(M^7,\bbz)= \bbz^2$. Item (5) follows from Poincar\'e duality and the fact that $M^7$ admits a Sasakian structure.

\end{proof}

Actually we have

\begin{lemma}\label{H30}
$H^3(M^7,\bbz)=0$.
\end{lemma}

\begin{proof}
First by (4) of Lemma \ref{exhomseq} $H^3(M^7,\bbz)$ is torsion free, so it suffices to work with $\bbq$ coefficients. Since $M^7$ is simply connected  and is an $S^1$ orbibundle over a stage 3 Bott orbifold $(M_3(a,b,c),\grD_\bfm)$ we can apply the Leray-Serre Theorem with $\bbq$ coefficients. The differential $d_2:E_2^{0,1}\ra{1.8} E_2^{2,0}=H^2(M^7,\bbq)$ sends the class $\gra$ of the fiber $S^1$ to the K\"ahler class $c_1x_1+c_2x_2+c_3x_3$ where $c_i\in\bbq^+$. So by naturality we have $d_2(\gra\otimes x_i)=(c_1x_1+c_2x_2+c_3x_3)x_i$. Suppose there would exist a class $w=w_1x_1+w_2x_2+w_3x_3\in E_2^{2,0}$ such that $d_2(\gra\otimes w)=0$. Then the 3-class $\gra\otimes w$ would survive to the limit giving a nonzero element in $H^3(M^7,\bbq)$ by the Leray-Serre Theorem. Now the cohomology ring of $M_3(a,b,c)$ is \cite{ChMasu10}
$$\bbz[x_1,x_2,x_3]/\bigl(x_1^2,x_2(ax_1 + x_2), x_3(bx_1 + cx_2 + x_3)\bigr).$$
So computing the $d_2$ differential we have
\begin{eqnarray*}
0&=&d_2(\gra\otimes w)=d_2(\gra)\otimes w=(c_1x_1+c_2x_2+c_3x_3)(w_1x_1+w_2x_2+w_3x_3)\\
  &=&(c_1w_2+c_2w_1-c_2w_2a)x_1x_2+(c_1w_3+c_3w_1-c_3w_3b)x_1x_3 +(c_2w_3+c_3w_2-c_3w_3c)x_2x_3
\end{eqnarray*}
which gives
\begin{equation}\label{H3soln}
\begin{pmatrix}
c_2 & c_1-c_2a & 0 \\
c_3 & 0 & c_1-c_3b \\
0 & c_3 & c_2-c_3c
\end{pmatrix}
\begin{pmatrix}
w_1 \\
w_2 \\
w_3
\end{pmatrix}=0.
\end{equation}
Since the coefficients $c_1,c_2,c_3$ are all positive, the rank of the matrix 
$$C=\begin{pmatrix}
c_2 & c_1-c_2a & 0 \\
c_3 & 0 & c_1-c_3b \\
0 & c_3 & c_2-c_3c
\end{pmatrix}$$
is either 3 or 2 which can be seen by putting $C$ in Jordan canonical form. If the rank of $C$ were 2 there would be exactly one solution to Equation \eqref{H3soln} which is an element of $E_2^{1,2}$ and since $E_2^{3,0}=H^3(M_3(a,b,c),\bbq)=0$, there would be precisely one generator in $H^3(M^7,\bbq)$ which contradicts the fact that $b_3(M^7)$ is even.
\end{proof}

\noindent{Proof of Theorem \ref{maintopthm}:} 
It follows from Lemmas \ref{exhomseq} and \ref{H30} and the fact that the homology (cohomology) of a connected sum is the direct sum of the homology (cohomology) of the two factors that $M^7$ has the rational homology (cohomology) of the connected sum $(S^2\times S^5)\# (S^2\times S^5)$. By Poincar\'e duality and Lemma \ref{exhomseq} it follows that $H^5(M^7,\bbz)$ and $H^6(M^7,\bbz)$ have no torsion. 
So it suffices to compute the torsion in $H^4(M^7,\bbz)$. First
\begin{lemma}\label{orbcoh}
We have
$$H^{r}_{orb}((\calh_a,\grD_\bfm),\bbz)=\begin{cases} \bbz &~~\text{if $r=0$} \\
                                     \bbz^2&  ~~\text{if $r=2$} \\
                                  \bbz\oplus \bbz_{m^0}\oplus \bbz_{m^\infty} & ~~\text{if $r=4$} \\
                                  \bbz_{m^0}\oplus \bbz_{m^\infty} & ~~\text{if $r=6,8,\cdots$} \\
                                  0 & ~~\text{if $r$ is odd.}
                                  \end{cases}$$
\end{lemma}

\begin{proof}
The Leray sheaf of the map 
$$p:\mathsf{B}(\calh_a,\grD_\bfm)\ra{2.5} ~\calh_a$$ 
is the derived functor sheaf $R^sp\bbz$,  that is, the sheaf associated to the presheaf $U\mapsto H^s(p^{-1}(U),\bbz)$. For $s>0$ the stalks of $R^sp\bbz$ at points of $U$ vanish if $U$ lies in the regular locas of $(\calh_a,\grD_\bfm)$ which is the complement of the union of the zero $e_0$ and infinity $e_\infty$ sections of the natural projection $\calh_a\ra{1.6} \bbc\bbp^1$. However, at points of $e_0$ and  $e_\infty$ the fibers of $p$ are (up to homotopy) the Eilenberg-MacLane spaces $K(\bbz_{m^0},1)$ and $K(\bbz_{m^\infty},1)$, respectively. So at points of $e_0(e_\infty)$ the stalks are the group cohomology $H^s(\bbz_{m^0},\bbz)\bigl(H^s(\bbz_{m^\infty},\bbz)\bigr)$. This is $\bbz$ for $s=0$ and $\bbz_{m^0}(\bbz_{m^\infty})$ at points of $e_0(e_\infty)$ when $s>0$ is even; it vanishes when $s$ is odd. The $E_2$ term of the Leray spectral sequence of the map $p$ is 
$$E_2^{r,s}=H^r(\calh_a,R^sp\bbz)$$
and by Leray's theorem this converges to the orbifold cohomology  $H^{r+s}_{orb}((\calh_a,\grD_\bfm),\bbz)$. Now $E_2^{r,0}=H^r(\calh_a,\bbz)$ and $E_2^{r,s}=0$ for $r$ or $s$ odd. For $r=0$ since $R^sp\bbz$ has its support in the orbifold singular locus $e_0\cup e_\infty$, the only continuous section of $R^sp\bbz$ is the 0 section which implies that $E_2^{0,s}=0$ for all $s$. Now we have $E_2^{2r,2s}=0$ for $r>1$ and  
$$E_2^{2,2s}=H^2(\calh_a,R^{2s}p)=H^2(e_0,\bbz_{m^0})\oplus  H^2(e_\infty,\bbz_{m^\infty})=\bbz_{m^0}\oplus \bbz_{m^\infty}.$$ 
One easily sees this spectral sequence collapses whose limit is the orbifold cohomology $H^{r}_{orb}((\calh_a,\grD_\bfm),\bbz)$ which implies the result. 
\end{proof}

To continue the proof of Theorem \ref{maintopthm}, as in \cite{BoTo14a,BoTo14NY}, we use the commutative diagram of fibrations
\begin{equation}\label{orbifibrationexactseq}
\begin{matrix}
Y^{p,q}\times S^3_\bfw &\ra{2.6} & M_{l_1,l_2,\bfw}& \ra{2.6}& \mathsf{B}S^1 \\
\decdnar{=}&&\decdnar{}&&\decdnar{\psi}\\
Y^{p,q}\times S^3_\bfw&\ra{2.6} &\mathsf{B}(\calh_a,\grD_\bfm) \times\mathsf{B}\bbc\bbp^1[\bfw]&\ra{2.6}
&\mathsf{B}S^1\times \mathsf{B}S^1.\, 
\end{matrix} 
\end{equation}
Here $\mathsf{B}G$ is the classifying space of a group $G$ or Haefliger's classifying space \cite{Hae84} of an orbifold if $G$ is an orbifold.
The lower exact fibration is a product of fibrations. We denote the orientation classes of $Y^{p,q}\times S^3=S^2\times S^3\times S^3$ by $\gra,\grb,\grg$, respectively. As in 3.2.2 of \cite{BoTo14NY} we have $d_4(\grg)=w_1w_2s_2^2$. For the fibration 
$$Y^{p,q}\ra{2.6}\mathsf{B}(\calh,\grD_\bfm)\ra{2.6}\mathsf{B}S^1$$ 
we have $d_4(\grb)=m^2v^0v^\infty s_1^2$. The map $\psi$ in Diagram \eqref{orbifibrationexactseq} is induced by the map $e^{i\theta}\mapsto (e^{il_2\theta}, e^{-il_1\theta})$, so $\psi^*s_1=l_2s$ and $\psi^*s_2=-l_1s$ the result follows by the commutativity of Diagram \eqref{orbifibrationexactseq}. \hfill $\Box$

\begin{remark}
Although the case $(p,q)=(1,0)$ does not fit directly into the $Y^{p,q}$ scheme of \cite{GMSW04a}, it can, nevertheless, be identified with the homogeneous Sasaki-Einstein structure on $S^2\times S^3$. Then if we take $(v^0,v^\infty)=(w_1,w_2)=(l_1,l_2)=(1,1)$ and $m=1$ we obtain the homogeneous Sasaki-Einstein structure on an $S^1$ bundle over $\bbc\bbp^1\times\bbc\bbp^1\times\bbc\bbp^1$ which is a 7-manifold with the integral cohomology of the 2-fold connected sum $2(S^2\times S^5)$. See Sections 11.1.1 and 11.4.2 of \cite{BG05} for details. The general $S^3_\bfw$ join with $Y^{1,0}$ was treated in Section 3.2.2 of \cite{BoTo14NY}.
\end{remark}

\section{A Generalized Orbifold Calabi construction}
We now discuss a family of explicit examples or orbifold K\"ahler-Einstein metrics that we may view as arising from a special case of the  generalized Calabi construction as presented in \cite[Section
  2.5]{ACGT04} and further discussed in \cite[Section 2.3]{ACGT11} and \cite[Section 5.1]{BoCaTo17} - here generalized to allowing certain mild orbifold singularities. 

The base of the construction will in this case be a K\"ahler-Einstein Hirzebruch orbifold $(\calh_a,\grD_{\bfm_2})$. This is the K\"ahler quotient of the quasi-regular Sasaki-Einstein $Y^{p,q}$ examples produced in \cite{GMSW04a}). Now $Y^{p,q}$ may be viewed as the total space of a $S^1$ principal orbi-bundle, $P$, over $(\calh_a,\grD_{\bfm_2})$ defined by the class
$\frac{c_1^{orb}(\calh_a,\grD_{\bfm_2})}{ \cali_{\bfv_2}}=\frac{(2m_2v^0_2+a)v^\infty_2 x_1+(v^0_2+v^\infty_2)x_2}{m_2v^0_2v^\infty_2 \cali_{\bfv_2}}$, where the notation is as in Lemma \ref{Ypqind}. In particular, the index of $(\calh_a,\grD_{\bfm_2})$ is $\cali_{\bfv_2}=\gcd(2m_2v^0_2+a)v^\infty_2,v^0_2+v^\infty_2)$, where $(m_2^0,m_2^\infty)=m_2(v_2^0,v_2^\infty)$ and $v_2^0,v_2^\infty$ are coprime.
Let $g_{base}$ denote the K\"ahler-Einstein orbifold metric whose K\"ahler form, $\omega_{base}$, satisfies that
$$\left[ \frac{\omega_{base}}{2\pi}\right] = \frac{(2m_2v^0_2+a)v^\infty_2 x_1+(v^0_2+v^\infty_2)x_2}{m_2v^0_2v^\infty_2\cali_{\bfv_2}}.$$
As we saw in Example 6.8 of \cite{BoTo14a}, the metric $g_{base}$ is explicit and "admissible" in the sense of \cite{ACGT08}. Note that the Ricci form of $g_{base}$ is given by $\rho_{base} = \cali_{\bfv_2} \omega_{base}$. 

We consider the generalized Calabi construction of orbifold K\"ahler metrics on the bunde $P \times_{S^1} \bbc\bbp^1_{m_3^0,m_3^\infty} \rightarrow (\calh_a,\grD_{\bfm_2})$. This may also be viewed as an admissible construction - extended to mild orbifold cases.

\begin{definition}\label{calabidata}
{\em Generalized orbifold Calabi data for our purposes}.
\begin{enumerate}
\item A log pair $(\calh_a,\grD_{\bfm_2})$ with K\"ahler-Einstein structure
  $(\omega_{base},g_{base})$ such that \newline
  $\left[ \frac{\omega_{base}}{2\pi}\right] = \frac{(2m_2v^0_2+a)v^\infty_2 x_1+(v^0_2+v^\infty_2)x_2}{m_2v^0_2v^\infty_2\cali_{\bfv_2}}.$
\item The weighted projective line $(\bbc \bbp ^1_{m_3^0,m_3^\infty} = \bbc \bbp ^1_{v_3^0,v_3^\infty}/\bbz_{m_3},g_{\bfm_3},\omega_{\bfm_3})$ with rational Delzant polytope $[-1,1] 
\subseteq  \bbr^*$ and momentum map $\gz\colon \bbc \bbp ^1_{m_3^0,m_3^\infty} \rightarrow [-1,1]$. Here $(m_3^0,m_3^\infty)=m_3(v_3^0,v_3^\infty)$ and $v_3^0,v_3^\infty$ are coprime.
\item A principal $S^1$ orbi-bundle, $P_n \rightarrow (\calh_a,\grD_{\bfm_2})$, with a principal
  connection of curvature $n\,\omega_{base}\in \Omega^{1,1}((\calh_a,\grD_{\bfm_2}),
\bbr)$, where $S^1$ acts  on $\bbc \bbp ^1_{m_3^0,m_3^\infty}$, $n\in {\bbz} \setminus \{0\}$, and $\gcd(n,m_3)=1$. Note that $n \in \,span_\bbz \{ v_3^0,v_3^\infty\}$ (since $v_3^0,v_3^\infty$ are coprime), so $m_3n \in \,span_\bbz \{ m_3^0,m_3^\infty\}$. 
\item A constant $0< |r_3| <1$ of same sign as $n$ \newline
{\rm [ensuring that the $(1,1)$-form $(1/r_3 + \gz)n\,\omega_{base}$ is positive for $\gz \in [-1,1]$]}.
\end{enumerate}
\end{definition}
From this data we may define the orbifold
\[
M_3=P_n\times_{S^1}\bbc \bbp ^1_{m_3^0,m_3^\infty} = \mathring{M_3}\times_{\bbc^*}\bbc \bbp ^1_{m_3^0,m_3^\infty}  \rightarrow  (\calh_a,\grD_{\bfm_2}),
\]
where $ \mathring{M_3} = P_n \times_{S^1}\left( \gz^{-1}(-1,1)\right)$. Since the curvature $2$-form
of $P_n$ has type $(1,1)$, $\mathring{M_3}$ is a holomorphic principal
$\bbc^*$ bundle with connection $\theta \in
\Omega^1(\mathring{M_3},\bbr)$ and $M_3$ is a complex orbifold.

On $\mathring{M_3}$ we define K\"ahler structures of the form
\begin{equation}\label{compatiblemetric}
\begin{aligned}
g &=(1/r_3 + \gz)n\,g_{base} + \frac{1}{\Theta(\gz)}d\gz^2 + \Theta(\gz) \theta^2\\
\omega &=  (1/r_3 + \gz)n\,\omega_{base} + d\gz \wedge \theta\\
d\theta &=n\,\omega_{base},
\end{aligned}
\end{equation}
where $\frac{1}{\Theta(\gz)} = \frac{d^2 U}{d \gz^2}$ and $U$ is the symplectic
potential \cite{Gui94b} of the chosen toric K\"ahler structure $g_{\bfm_3}$ on $\bbc \bbp ^1_{m_3^0,m_3^\infty}$.

The {\em generalized Calabi construction} arises from seeing
\eqref{compatiblemetric} as a blueprint for the construction of various orbifold
K\"ahler metrics on $M_3$ by choosing various smooth functions $\Theta(\gz)$ on $(-1,1)$ satisfying that
\begin{itemize}
\item \textup{[boundary values]} 
the following endpoint conditions are satisfied
\begin{equation}\label{eq:toricboundary}
\Theta(\pm 1) =0\qquad {\rm and}\qquad \Theta\,'(-1) =2/m_3^\infty\qquad {\rm and}\qquad \Theta\,'(1)=-2/m_3^0;
\end{equation}
\item \textup{[positivity]} 
the function $\Theta(\gz)$ is positive for $\gz\in(-1,1)$.
\end{itemize}
Then \eqref{compatiblemetric} extends to an orbifold K\"ahler metric on $(M(a,b,c),\Delta_\bfm)$, where
$b = n \frac{(2m_2v^0_2+a)v^\infty_2}{\cali_{\bfv_2}}$, $c=n\frac{v^0_2+v^\infty_2}{\cali_{\bfv_2}}$, and
$\bfm = (1,1,m_2^0,m_2^\infty,m_3^0,m_3^\infty)$.
Metrics
constructed this way are called {\em compatible K\"ahler metrics} with {\em compatible K\"ahler classes} parametrized by
$r_3$.
From \cite{ApCaGa06} (and specifically to our notation from \cite[Proposition 5.4]{BoTo14a}) we have that
the compatible metric defined by $\Theta(\gz)$ is K\"ahler-Einstein exactly when
\begin{equation}\label{KE1}
2r_3\cali_{\bfv_2}/n = (1+r_3)/m_3^\infty + (1-r_3)/m_3^0
\end{equation}
and
\begin{equation}\label{KE2}
\int_{-1}^1 ((1-\gz)/m_3^\infty - (1+\gz)/m_3^0)(1+r_3 \gz)^2\,d\gz=0.
\end{equation}
\section{Explicit Sasaki-Einstein metrics}

We now look more closely for explicit Sasaki-Einstein examples arising from the join from Theorem \ref{38thm}. 
The arguments in Sections 6.1 and 6.2 of \cite{BoTo14a} (see specifically page 1053) carry through so that we have an adapted version of Theorem 1.4 of \cite{BoTo14a}:

\begin{theorem}\label{SEexistence}
Consider the join $M^7=Y^{p,q}\star_{l_1,l_2}S^3_\bfw$ where $Y^{p,q}$ has a quasi-regular Sasaki-Einstein structure
with quotient Hirzebruch orbifold $(\calh_a,\grD_{\bfv_2})$ with $a>0$
and $l_1,l_2$ are given by
$$l_1=\frac{\cali_{\bfv_2}}{\gcd(|\bfw|,\cali_{\bfv_2})},\qquad l_2=\frac{|\bfw|}{\gcd(|\bfw|,\cali_{\bfv_2})}.$$
Then for each vector $\bfw = (w_1,w_2) \in \bbz^+\times\bbz^+$ with relatively prime components satisfying $w_1>w_2$ there exist a Reeb vector field $\xi_{\bfv_3}$ in the $2$-dimensional $\bfw$-Sasaki cone on $M^7$ such that the corresponding Sasakian structure is Sasaki-Einstein.
\end{theorem}

Specifically, using equation (59) from \cite{BoTo14a}, we know that if the ray defined by co-prime $(v_3^0,v_3^\infty)$ is quasi-regular, then
we ought to look at the K\"ahler class determined by
$r_3=\frac{w_1v_3^\infty - w_2v_3^0}{w_1v_3^\infty + w_2v_3^0}$.
Now with this $r_3$, and the above choice of $(l_1,l_2)$, \eqref{KE1} is automatically solved and \eqref{KE2} becomes (similarly to (68) in \cite{BoTo14a})
\begin{equation}\label{SEall}
\int_{-1}^1 \left((v_3^0-v_3^\infty)-(v_3^0+v_3^\infty)\gz) \right)((w_1v_3^\infty+w_2v_3^0)+ (w_1v_3^\infty-w_2v_3^0)\gz)^2d\gz = 0.
\end{equation}
This equation defines a priori a quasi-regular Sasaki $\eta$-Einstein ray (and thus, up to transverse homothety, a Sasaki-Einstein structure), but, by the same arguments as in Section 6.1 of \cite{BoTo14a}, any solution $(v_3^0,v_3^\infty) \in \bbr^+\times \bbr^+$ of \eqref{SEall} defines a Sasaki $\eta$-Einstein ray in the $\bfw$-cone, which is irregular unless $v_3^\infty/v_3^0 \in \bbq$. This gives

\begin{corollary}\label{SEsoln}
For any pair of relatively prime positive integers $(w_1,w_2)$ satisfying $w_1>w_2$, we obtain an SE metric by solving the cubic equation
\begin{equation}\label{kw}
3w_2k^3+(2w_2-w_1)k^2-(2w_1-w_2)k-3w_1=0.
\end{equation}
The ray of the Reeb vector field of the SE metric is then given by
\begin{equation}\label{kv}
\frac{v_3^\infty}{v_3^0} =  \frac{3+2k+k^2}{(1+2k+3k^2)}=k\frac{w_2}{w_1}.
\end{equation}
The SE metric is quasi-regular when $k$ is a rational root of \eqref{kw} and irregular when it is an irrational root.
\end{corollary}

It follows from the analysis in Section 6.2 of \cite{BoTo14a} that the positive real root $k$ lies in the open interval $(1,\infty)$.

\subsection{Quasi-regular examples}
Changing our point of view we see that for any $k\in (1,+\infty)\cap \bbq$ we can choose $(w_1,w_2)$ such that
\begin{equation}\label{kw2}
w_2/w_1 = \frac{3+2k+k^2}{k(1+2k+3k^2)}
\end{equation}
and then the Sasaki-Einstein metric from Theorem \ref{SEexistence} is quasi-regular with ray $\bfv_3$ defined by Equation \eqref{kv}. We now give some examples.

\begin{example}\label{138,7}
Here we give an example that builds on a bouquet of Sasaki cones from Example\footnote{Example 6.8 in this reference has a small typo, namely $v_2^0$ and $v_2^\infty$ got switched.} 6.8 of \cite{BoTo14a}. Corollary 5.5 of \cite{BoPa10} describes the well known $Y^{p,q}$ structures \cite{GMSW04a} on $S^2\times S^3$ as a $|\phi(p)|$-bouquet of Sasaki cones where $|\phi(p)|$ denotes the order of the Euler phi function $\phi(p)$. Let us consider the example when $p=13$. Since $13$ is prime the bouquet consists of $p-1=12$ Sasaki cones labeled by the $12$ positive integers $1\leq q<13$ and as such contains $12$ Sasaki-Einstein metrics. However, in order to construct SE metrics on our 7-manifolds, we need the SE structure on $Y^{p,q}$ to be quasi-regular. It is easy to check that for  the bouquet $\bigcup\{Y^{13,q}\}_{q=1}^{12}$ the only values of $q$ where the SE metric is quasi-regular is for $q=7,8$, all the other SE metrics in the bouquet are irregular. Let us look at these two cases a bit closer. 

In the case $Y^{13,8}$ we have
$a=70$, $v_2^0= 7$, $v_2^\infty=5$, $m_2=13$, so $m_2^0=91$ and $m_2^\infty=65$. This gives us that
$\cali_{\bfv_2} = 12$. If we put $k=2$ in the quasi-regular SE prescription above given by Equation \eqref{kw}, we see that $(w_1,w_2)=(34,11)$, so
$l_1=4$ and $l_2=15$, and $(v_3^0,v_3^\infty)=(17, 11)$. Now we calculate that
$s=\gcd(|w_1 v_3^\infty-w_2 v_3^0|,l_2)=1$ so
$n=l_1\frac{w_1 v_3^\infty-w_2 v_3^0}{s}= 748$, and 
$(m_3^0,m_3^\infty)=15(17,11)$. Then the quotient is the Bott orbifold given by the log pair $(M(a,b,c),\Delta_\bfm)$, where 
$$
\begin{array}{ccl}
a &=& 70,\\
b &= & n \hat{b} = n\frac{(2m_2v_2^0+a)v_2^\infty}{\cali_{\bfv_2}} =78540,\\
c &= & n \hat{c}=n\frac{v_2^0+v_2^\infty}{\cali_{\bfv_2}} =748,\\ 
\bfm &= &(1,1,91,65,255,165).
\end{array}
$$ 

In the case where $p=7$ in $Y^{p,q}$ we have
$a=36$, $v_2^0= 4$, $v_2^\infty=3$, $m_2=13$, so $m_2^0=52$ and $m_2^\infty=39$. This gives us that
$\cali_{\bfv_2} = 7$. Again we put $k=2$ in Equation \eqref{kw} which gives 
$(w_1,w_2)=(34,11)$ and $(v_3^0,v_3^\infty)=(17, 11)$ respectively. Now $l_1=7$ and $l_2=45$, and 
$s=\gcd(|w_1 v_3^\infty-w_2 v_3^0|,l_2)=1$ so
$n=l_1\frac{w_1 v_3^\infty-w_2 v_3^0}{s}= 1309$, and 
$(m_3^0,m_3^\infty)=45(17,11)$. 
Then the quotient is the Bott orbifold given by the log pair $(M(a,b,c),\Delta_\bfm)$, where 
$$
\begin{array}{ccl}
a &=& 36,\\
b &= & n \hat{b} = n\frac{(2m_2v_2^0+a)v_2^\infty}{\cali_{\bfv_2}} =78540,\\
c &= & n \hat{c}=n\frac{v_2^0+v_2^\infty}{\cali_{\bfv_2}} = 1309,\\ 
\bfm &= &(1,1,52,39,765,495).
\end{array}
$$ 

One can easily check from the torsion in Theorem \ref{maintopthm} that the two SE 7-manifolds $Y^{13,8}\star_{4,15}S^3_{(34,11)}$ and $Y^{13,7}\star_{7,45}S^3_{(34,11)}$ are not homotopy equivalent. For both of these 7-manifolds the Reeb field that gives the SE metric is quasi-regular. Moreover, they are both induced from the same Reeb ray, namely $\{\xi_{a(17,11)}\}_{a>0}$ of the same $S^3_{(34,11)}$, and $b$ is the same for the quotient orbifolds. 

For each choice of the rational number $k>1$ we obtain a pair of quasi-regular SE 7-manifolds induced by the $S^3_\bfw$ join and its Reeb field $\xi_{\bfv_3}$ where $\bfw$ and $\bfv_3$ are determined by Equations \eqref{kw} and \eqref{kv}.
\end{example}

\begin{example} Here we give a $1$-parameter family of smooth quasi-regular examples.
First, let $k_2 \in \bbz^{\geq 0}$ be given (using the subscript "$2$" to indicate that this is a choice at the second stage). Then from \cite{GMSW04a}) we get a quasi-regular Sasaki-Einstein $Y^{p,q}$ example by choosing
$$p=12k_2^2+18k_2+7 \quad \quad \quad {\rm and} \quad \quad \quad q=12k_2^2 + 16 k_2+5.$$
It is not hard to check that $\gcd(p,q)=1$ and $\gcd(p+q,p-q)=2$. Accordingly we recognize from \eqref{Ypqjoin} that
$$
Y^{p,q}=S^3\star_{l,p}S^3_{\frac{p+q}{l},\frac{p-q}{l}}= S^3\star_{2,(12k_2^2+18k_2+7)}S^3_{(2+3k_2)(3+4k_2),(1+k_2)}
$$
and that the quotient Hirzebruch orbifold of the quasi-regular Sasaki-Einstein metric is $(\calh_a,\grD_{\bfm_2})$ where,
using  \eqref{YpqisEinstein} and \eqref{step1data},
$$
\begin{array}{ccl}
{\bfm}_2 & = & (m^0_2,m^\infty_2)=(12k_2^2+18k_2+7)((3+4k_2),2(1+k_2)),\\
\\
m_2& = & 12k_2^2+18k_2+7,\\
\\
a & = &6(1+k_2)(1+2k_2)(3+4k_2).
\end{array}
$$
We can calculate from Lemma \ref{Ypqind} that $\cali_{\bfv_2}  = 5+6k_2$.

Now choosing $k=3$ in \eqref{kw} and \eqref{kv},
we have, from  Theorem \ref{SEexistence}, a quasi-regular Sasaki-Einstein structure on 
$M^7=Y^{p,q}\star_{l_1,l_2}S^3_\bfw$
with quotient log pair $(M_3(a,b,c),\grD_\bfm)$ given by Theorem \ref{38thm}.
Indeed, we have
$$
\begin{array}{ccl}
a&=& 6(1+k_2)(1+2k_2)(3+4k_2)\\
\\
b&=&4(1+k_2)(2+3k_2)(3+4k_2)n\\
\\
c&=&n\\
\\
\bfm &=& (m_1^0,m_1^\infty, m_2^0,m_2^\infty,m_3^0,m_3^\infty)\\
\\
(m_1^0,m_1^\infty) &=& (1,1)\\
\\
(m_2^0,m_2^\infty) &=&(12k_2^2+18k_2+7)((3+4k_2),2(1+k_2))\\
\\
(m_3^0,m_3^\infty) &=&m_3 (v_3^0, v_3^\infty),
\end{array}
$$
where
$$
\begin{array}{ccl}
n  & =  & l_1\frac{102}{\gcd(102,l_2)}\\
\\
m_3 & = & \frac{l_2}{\gcd(102,l_2)}\\
\\
l_1 & =& \frac{\cali_{\bfv_2}}{\gcd(20,5+6k_2)}\\
\\
l_2 & = & \frac{20}{\gcd(20,5+6k_2)}\\
\\
(w_1,w_2) &=& (17,3) \\
\\
(v_3^0,v_3^\infty)&= &(17,9).
\end{array}
$$

Using Lemma \ref{smoothquot}, we know that
the corresponding Sasaki structure is a smooth manifold if and only if 
$$
\begin{array}{rcl}
\gcd(l_2(12k_2^2+18k_2+7)(3+4k_2),17l_1)&=&1\\
\gcd(l_2((12k_2^2+18k_2+7)(3+4k_2),3l_1)&=&1\\
\gcd(l_2(12k_2^2+18k_2+7))2((1+k_2)),17l_1)&=&1\\
\gcd(l_2(12k_2^2+18k_2+7)2(1+k_2),3l_1)&=&1
\end{array}
$$
In order to get smooth examples, let us now assume $k_2=255\,t+ 10$ with $t \in \bbz^{+}$. Then we have
$$
\begin{array}{ccl}
\cali_{\bfv_2} & =& 5(306\,t+ 13)\\
\\
l_1 & =& 306\,t+ 13\\
\\
l_2 & = & 4\\
\\
n  & =  & 51(306\,t+ 13) \\
\\
m_3 & = & 2\\
\
\end{array}
$$
and then the corresponding Sasaki structure is a smooth manifold if and only if 
$$
\begin{array}{rcl}
\gcd(4(1387+ 65790 t+780300 t^2)(1020\,t+43),17(306\,t+ 13))&=&1\\
\gcd(4(1387+ 65790 t+780300 t^2)(1020\,t+43),3(306\,t+ 13))&=&1\\
\gcd(8(1387+ 65790 t+780300 t^2))(255\,t+ 11),17(306\,t+ 13))&=&1\\
\gcd(8(1387+ 65790 t+780300 t^2))(255\,t+ 11),3(306\,t+ 13))&=&1.
\end{array}
$$
if and only if
$$
\begin{array}{rcl}
\gcd((1387+ 65790 t+780300 t^2)(1020\,t+43),(306\,t+ 13))&=&1\\
\gcd((1387+ 65790 t+780300 t^2))(255\,t+ 11),(306\,t+ 13))&=&1.
\end{array}
$$
Since $\forall t\in \bbz^+$,
$$
\begin{array}{rcl}
6(255\,t+ 11)-5(306\,t+ 13)&=&1\\
\\
10(306\,t+ 13)-3(1020\,t+43) &=&1\\
\\
3 \left(780300 t^2+65790 t+1387\right)-(7650 t+320) (306 t+13)&=&1,
\end{array}
$$
we have that this is always satisfied.
Note that with $k_2=255\,t+10$ we get
$$
\begin{array}{ccl}
a&=& 6(255\,t+ 11)(510\,t+21)(1020\,t+43)\\
\\
b&=&204(255\,t+ 11)(765\,t + 32)(1020\,t+43)(306\,t+ 13)\\
\\
c&=&51(306\,t+ 13)\\
\\
\bfm &=& (m_1^0,m_1^\infty, m_2^0,m_2^\infty,m_3^0,m_3^\infty)\\
\\
(m_1^0,m_1^\infty) &=& (1,1)\\
\\
(m_2^0,m_2^\infty) &=&(1387+ 65790 t+780300 t^2)(1020\,t+43, 2(255\,t+ 11))\\
\\
(m_3^0,m_3^\infty) &=& 2 (17,9).
\end{array}
$$
Finally, the $p$ and $q$ in $Y^{p,q}$ are here given by
$$p= 780300 t^2+65790 t+1387$$
and 
$$q= 15 (170 t+7) (306 t+13),$$
so the smooth Sasaki Einstein structures live on
$$Y^{780300 t^2+65790 t+1387,15 (170 t+7) (306 t+13)}\star_{306\,t+ 13,4} S^3_{17,3}.$$
\end{example}

\def\cprime{$'$} \def\cprime{$'$} \def\cprime{$'$} \def\cprime{$'$}
  \def\cprime{$'$} \def\cprime{$'$} \def\cprime{$'$} \def\cprime{$'$}
  \def\cdprime{$''$} \def\cprime{$'$} \def\cprime{$'$} \def\cprime{$'$}
  \def\cprime{$'$}
\providecommand{\bysame}{\leavevmode\hbox to3em{\hrulefill}\thinspace}
\providecommand{\MR}{\relax\ifhmode\unskip\space\fi MR }

\providecommand{\MRhref}[2]{%
  \href{http://www.ams.org/mathscinet-getitem?mr=#1}{#2}
}
\providecommand{\href}[2]{#2}

\end{document}